\documentclass{article}
\usepackage[hscale=0.7,vscale=0.8]{geometry}
\usepackage{graphicx}%
\usepackage{multirow}%
\usepackage{amsmath,amssymb,amsfonts}%
\usepackage{mathrsfs}%
\usepackage[title]{appendix}%
\usepackage{xcolor}%
\usepackage{textcomp}%
\usepackage{manyfoot}%
\usepackage{booktabs}%
\usepackage{algorithm}%
\usepackage{algorithmicx}%
\usepackage{algpseudocode}%
\usepackage{listings}%
\usepackage[english]{babel}
\usepackage{mathptmx}     
\usepackage{latexsym}
\usepackage{relsize}
\usepackage{cancel}
\usepackage{hyperref}
\usepackage{amsthm}
\usepackage[affil-it]{authblk}

\newtheorem{teo}{Theorem}
\newtheorem{lem}{Lemma}
\newtheorem{rem}{Remark}

\newtheorem{coro}{Corollary}

\begin{document}
	
	\title{New tools for the study of Bochner differential operators}
	
	\author{L.M. Anguas\thanks{The work of L.M. Anguas was partially supported by Agencia Estatal de Investigación, Ministerio de Ciencia e Innovación, Spain, under grant PID2019-106362GB-I00.}}
	\affil{Saint Louis University, Madrid Campus, Avenida del Valle 34, 2803 Madrid, Spain\\\texttt{luismiguel.anguas@slu.edu}}
	
	\author{D. Barrios Rolan\'{\i}a\thanks{ The work of D. Barrios Rolan\'{\i}a was partially supported by Agencia Estatal de Investigación, Ministerio de Ciencia e Innovación, Spain, under grant PID2021-122154NB-I00.}}
	\affil{D. Barrios Rolan\'{\i}a ETSI Industriales, Universidad Polit\'ecnica de Madrid, C/Jos\'e Guti\'errez Abascal 2, 28006 Madrid, Spain\\\texttt{dolores.barrios.rolania@upm.es}}
	
	\date{}

	\maketitle
	
	\begin{abstract}
		A sequence $\{\delta_n^{(k)}\}$ associated to a Bochner differential operator is introduced as an effective tool to study this kind of operators. Some properties of this sequence are proven and used to deduce that a particular operator leads to solutions of a bispectral problem. In addition, the inverse problem is studied; that is, given a sequence 
		$\{\lambda_n\}$ of complex numbers and a sequence
		$\{P_n\}$ of polynomials with complex coefficients, 
		$\deg{P_n}=n$, we find a necessary and sufficient condition for the existence of a Bochner differential operator that has those sequences as eigenvalues and eigenpolynomials, respectively. The mentioned condition also depends on 
		$\{\delta_n^{(k)}\}$.

	\end{abstract}
	
	

	\section{Introduction}
	Throughout this work, we consider the ordinary differential operator of order $N$ defined as
	\begin{equation}
		L\equiv                
		\sum_{i=0}^Na_i(x)\partial_x^i
		\label{1}
	\end{equation}
	where 
	$a_i(x)
	$
	are complex polynomials in the variable $x$,
	$
	\deg (a_i)\leq i,
	$ and 
	$
	\partial^i_x,\, i=1,\ldots,N$, represents the derivative of order $i$ with respect to $x$. We call this operator, the \textit{Bochner differential operator}.
	We also consider a sequence of eigenfunctions
	$
	\{P_n\},
	$
	which we assume that are monic polynomials with $\deg(P_n)=n$ for each $n\in \mathbb{N}$. That is,
	\begin{equation}        
		\sum_{i=1}^Na_i(x)\partial_x^iP_n(x)=\lambda_nP_n(x)\,,\quad \forall n\in \mathbb{N},
		\label{11}
	\end{equation}
	where 
	$
	\{\lambda_n\}\subset \mathbb{C}
	$
	is the corresponding sequence of eigenvalues.
	The polynomials $\{P_n\}$ satisfying \eqref{11} are called {\em eigenpolynomials} in this work. 
	
	Although we study eigenpolynomials of \eqref{1} in general, we are especially interested in families of polynomials that are, at the same time, eigenfunctions of a certain difference operator $J$, satisfying
	\begin{equation}
		\label{discreto}
		\left(J(n)P_n\right)(x)=xP_n(x),
	\end{equation}
	where
	$$
	J(n)P_n=\sum_{k=n-p}^{n}\alpha_{n,k}P_k+P_{n+1}
	$$
	for a fixed $p\in\mathbb{N}$.
	This is equivalent to the fact that the polynomials 
	$\{P_n\}$  
	satisfy
	a $(p+2)$-term recurrence relation
	\begin{equation}
		\label{otro9}
		\left.
		\begin{array}{r}
			\displaystyle \sum_{k=n-p}^{n-1}\alpha_{n,k}P_k(x)+(\alpha_{n,n}-x)P_n(x)+P_{n+1} (x)= 0\,,\quad n=0,1,\ldots ,\\
			P_0=1,\quad P_{-1}=\cdots =P_{-p}=0.
		\end{array}
		\right\}
	\end{equation}
	If the polynomials
	$
	P_n
	$ 
	satisfy \eqref{11} and \eqref{discreto}, we say that the sequence
	$
	\{P_n\},\,n\in\mathbb N,
	$ 
	is a solution for the bispectral problem defined by $L$ and $J$. This kind of problem was studied for the first time by S. Bochner in 1929 \cite{Bochner}. He studied the case where the order of the differential operator is $N=2$, determining the sequences of classical orthogonal polynomials (Hermite, Jacobi and Laguerre) as the solutions of the bispectral problem with $p=1$ in \eqref{otro9}. Eleven years later, Krall \cite{Krall1}, \cite{Krall}, established a classification with seven families of eigenpolynomials in the case $N=4$. Since these relevant works of Bochner and Krall, several contributions were made about properties of eigenpolynomials of operators \eqref{1} and also about the bispectral problem (see, for example, \cite{Shapiro},  \cite{Shapiro2} \cite{Paco}). However, the study of the case of arbitrary order $N$ remains open. 
	
	In \cite{AngBar} we introduced the sequences 
	$\{\delta_n^{(k)}\},\,n\in \mathbb{N},$ for $k=0,1,\ldots,N,$ associated with the differential operator \eqref{1} of order $N$ as an auxiliary tool that helped us obtaining the explicit expression of the
	eigenvalues and eigenfunctions of a differential operator. The goal of this work is to present these sequences $\{\delta_n^{(k)}\}$ as a useful tool in the study of Bochner differential operators and its eigenpolynomials. In order to do this, we will first prove some properties that the sequences $\{\delta_n^{(k)}\}$ satisfy. In particular, we will obtain an expression for the coefficients of the eigenpolynomials in terms of a (small) amount of elements of the sequence $\{\delta_n^{(k)}\}$.
	
	Furthermore, we will use the sequence $\{\delta_n^{(k)}\}$ to prove that a particular Bochner differential operator of order $N$ leads to a sequence of eigenpolynomials satisfying an $(N+1)$-term recurrence relation. This operator appeared in \cite{Shapiro2} as one of the only two operators conjectured to give rise to sequences of eigenpolynomials satisfying an $(N+1)$-term recurrence relation; that is, as one of the only two operators whose eigenpolynomials are solutions of this bispectral problem. 
	
	Finally, we will use the sequence $\{\delta_n^{(k)}\}$ to obtain a necessary and sufficient condition for the existence of solution of the inverse problem. That is, given a sequence $\{\lambda_n\}$ of complex numbers and a sequence $\{P_n\}$ of polynomials with complex coefficients such that  
	$\deg{P_n}=n$ for each $n\in \mathbb{N}$, we will find a necessary and sufficient condition in terms of $\{\delta_n^{(k)}\}$ for the existence of a Bochner differential operator such that $\{\lambda_n\}$ and $\{P_n\}$ are its eigenvalues and eigenpolynomials, respectively.
	
	Before expanding on our main goals, we would like to introduce a result that explains the equivalence between the sequences 
	$\{\delta_n^{(k)}\}$
	and the coefficients of polynomials $a_i(x)$ defining \eqref{1}. 
	
	\begin{teo}
		\label{Lema2}
		Given two sequences $\{a_{n,i}\}, i=0,1,\dots,n,\, n\in \mathbb{N}$, and $\{\delta_n^{(k)}\},\ n=0,1,\dots,\ k=0,1,\dots,n,$ the following connections are equivalent
		\begin{enumerate}
			\item For each $n=0,1,\dots,$
			\begin{equation}
				\label{deltas_aes}
				\delta_n^{(k)}=\sum_{i=k}^{n} {n \choose i} i!a_{i,i-k}, \quad  k=0,1,\ldots, n\,.
			\end{equation}
			\item For each $n=0,1,\dots,$
			\begin{equation}
				\label{BB}
				n!a_{n,n-k}=\sum_{i=k}^{n} {n \choose i}(-1)^{n-i}\delta_i^{(k)}\quad \text{ for } k=0,1,\ldots, n\,.
			\end{equation}
		\end{enumerate}
	\end{teo}
	The proof of Theorem \ref{Lema2} is included in Section \ref{auxiliares} along with some auxiliary results used in this paper .
	
	Theorem \ref{Lema2} provides the key to define and use the sequences 
	$\{\delta_n^{(k)}\}$
	from the differential operator
	$L$
	given in \eqref{1}. Thus, in Section \ref{directo} the direct problem is addressed and, using \eqref{deltas_aes}, the eigenpolynomials and eigenvalues of $L$ are obtained. Moreover, in this section we will prove that, for each
	$
	k=0,1,\ldots, N,
	$ 
	the sequence
	$\{\delta_n^{(k)}\},\,n\in\mathbb{N},$
	is obtained from the $N-k+1$ first terms
	$\delta_k^{(k)},\,\delta_{k+1}^{(k)},\ldots, \delta_N^{(k)}.$
	
	In Section \ref{ejemplo}, we will use the tools developed in Section \ref{directo} to address the conjecture from \cite{Shapiro2} about the operators that lead to solutions of the bispectral problems.
	
	On the other hand, \eqref{BB} makes possible to define the coefficients 
	$a_i(x)$
	of the differential operator \eqref{1} using the sequences 
	$\{\delta_n^{(k)}\}$.
	This will be done in Section \ref{inverso}, where the inverse problem is approached. 
	
	To finish, we will summarize our conclusions and express some ideas about possible future works in Section \ref{Conclusiones}.

	\section{Auxiliary results }
	\label{auxiliares}
	
	In order to prove our main results, we need several lemmas related to Combinatorial Analysis. The proofs of lemmas \ref{Lema1} and \ref{Lema33} are simple exercises that are omitted. The rest of the proofs are not so simple and are included to make for easier reading.
	
	Here and in the rest of our paper, we assume
	$
	\displaystyle\binom r s =0
	$
	when 
	$r<s$  or
	$s<0.$ 
	
	\begin{lem}
		\label{Lema1}
		For $m=0,1,\ldots$ we have
		\begin{eqnarray}
			\sum_{r=k}^{m+1}(-1)^{r-k}\binom{r}{k}\binom{m+1}{r}=0,& k=0,1,\dots,m.\label{1_lema1}\\\nonumber\\
			\sum_{r=0}^m \binom{m+k+1}{r}(-1)^r=(-1)^m\displaystyle\binom{m+k}k, & k=0,1\dots.\label{2_lema1}
		\end{eqnarray}
	\end{lem}
	
	\begin{lem}
		For 
		\label{Lema33}
		$k,\,m\in \mathbb{N}$ we have
		\begin{equation}
			\label{kym}
			\sum_{j=0}^{k-1} \binom{k-1}{j}\frac{(-1)^j}{j+m+1}=\frac{1}{k\displaystyle\binom{m+k}k}.
		\end{equation}
	\end{lem}
	
	\begin{lem}
		\label{nuevo_lema2}
		For 
		$n\in\mathbb{N},\ q=0,1,\ldots, n$,
		and
		$r=0,1,\ldots,$
		we have
		\begin{equation}
			\label{8_tilde}
			\sum_{s=0}^{q}(-1)^{s}\binom{q}{s}\binom{n-q+s}{r}=
			\left\{
			\begin{array}{lll}
				0&,&\text{if } r<q\\\\
				(-1)^q\displaystyle\binom{n-q}{r-q}&,& \text{if }r\ge q.
			\end{array}
			\right.
		\end{equation}
	\end{lem}
	
	\begin{proof}
		Let
		$n\in\mathbb{N}$
		be fixed. We will prove \eqref{8_tilde} by induction on 
		$q\leq n$. 
		\begin{itemize}
			\item[(a)]
			For $q=0$, \eqref{8_tilde} is obvious. For
			$q=1$,
			$$\sum_{s=0}^{1}(-1)^{s}\binom{1}{s}\binom{n-1+s}{r}=\binom{n-1} r-\binom {n} r.$$
			
			If
			$r<q=1$, necessarily 
			$r=0$ and
			$\binom{n-1} r-\binom {n} r=0$. 
			
			If
			$r\ge q=1$ and $r\leq n$ then 
			$\binom{n-1} r-\binom {n} r=-\binom{n-1}{r-1}$,
			which is \eqref{8_tilde}.
			
			If
			$r\ge q=1$ and $r> n$ then also
			$r> n-1$
			and
			$\binom{n-1} r-\binom {n} r=0=-\binom{n-1}{r-1}$, which is \eqref{8_tilde} again.
			\item[(b)]
			We assume that \eqref{8_tilde} holds for 
			$q=1,2,\ldots, \tilde q$ with
			$\tilde q<n$. Then
			$\tilde q+1\le n$ and
			\begin{eqnarray}
				\label{nueve}
				&&\sum_{s=0}^{\tilde q+1}(-1)^{s}\binom{\tilde q+1}{s}\binom{n-\tilde q-1+s}{r}=
				\sum_{s=0}^{\tilde q+1}(-1)^{s}\left[  \binom{\tilde q}{s}+\binom{\tilde q}{s-1} \right]\binom{n-\tilde q-1+s}{r}\nonumber\\
				&=&\sum_{s=0}^{\tilde q}(-1)^{s}\binom{\tilde q}{s}\binom{n-\tilde q-1+s}{r}-\sum_{s=1}^{\tilde q+1}(-1)^{s-1}\binom{\tilde q}{s-1}\binom{n-\tilde q-1+s}{r}\nonumber\\
				&=&\sum_{s=0}^{\tilde q}(-1)^{s}\binom{\tilde q}{s}\binom{(n-1)-\tilde q+s}{r}-\sum_{s=0}^{\tilde q}(-1)^{s}\binom{\tilde q}{s}\binom{n-\tilde q+s}{r}
			\end{eqnarray}
			(where, as usual,
			$
			\binom{\tilde q}{j}=0
			$ 
			if $j>\tilde q$ or $j<0$).
			
			If
			$r\geq \tilde q+1$
			then also 
			$r\geq \tilde q$
			and applying \eqref{8_tilde} in both terms of \eqref{nueve} we obtain
			$$
			\sum_{s=0}^{\tilde q+1}(-1)^{s}\binom{\tilde q+1}{s}\binom{n-\tilde q-1+s}{r}=(-1)^{\tilde q+1}\binom{n-(\tilde q+1)}{r- (\tilde q +1)},
			$$
			which is  \eqref{8_tilde} substituting 
			$q$ by
			$\tilde q+1$.
			
			If
			$r< \tilde q$
			then, applying \eqref{8_tilde} to $\tilde q$, both terms of \eqref{nueve} are $0$ and we have 
			$$
			\sum_{s=0}^{\tilde q+1}(-1)^{s}\binom{\tilde q+1}{s}\binom{n-\tilde q-1+s}{r}=0.
			$$
			Finally, if 
			$r=\tilde q$, \eqref{nueve} is equal to
			$$
			\sum_{s=0}^{\tilde q+1}(-1)^{s}\binom{\tilde q+1}{s}\binom{n-\tilde q-1+s}{r}=(-1)^{\tilde q}-(-1)^{\tilde q}=0
			$$
			because
			$
			\displaystyle\binom{(n-1)-\tilde q}{r-\tilde q}=\binom{n-\tilde q}{r-\tilde q}=1.
			$
			
		\end{itemize}
	\end{proof}
	
	\begin{lem}
		\label{lema3}
		For 
		$k,\,n,\,m,\,r \in\mathbb{N},\ r\leq m\leq n $
		and
		$k\leq n-m$
		we have
		\begin{equation}
			\label{10_tilde}
			\sum_{s=1}^{m+1}\frac{s}{s+k}(-1)^{s}\binom{m+1}{s}\binom{n-m+s}{r}=
			-\frac{\displaystyle\binom{n-m-k}{r}}{\displaystyle\binom{m+k+1}{k}}
		\end{equation}
	\end{lem}
	
	\begin{proof}
		Writing 
		$
		\displaystyle\frac{s}{s+k}=1-\frac{k}{s+k}$, the left hand side of \eqref{10_tilde} is
		\begin{eqnarray}
			\label{once}
			&&\sum_{s=1}^{m+1}\frac{s}{s+k}(-1)^{s}\binom{m+1}{s}\binom{n-m+s}{r}\nonumber\\
			& = & \sum_{s=1}^{m+1}(-1)^{s}\binom{m+1}{s}\binom{n-m+s}{r}-k\sum_{s=1}^{m+1}\frac{(-1)^{s}}{s+k}\binom{m+1}{s}\binom{n-m+s}{r}.
		\end{eqnarray}
		Since 
		$
		r<m+1$,
		we can apply \eqref{8_tilde} (for $q=m+1$) and the first addend of \eqref{once} leads to

		\begin{eqnarray}
			\label{doce}
			&&\sum_{s=1}^{m+1}(-1)^{s}\binom{m+1}{s}\binom{n-m+s}{r}=
			\sum_{s=1}^{m+1}(-1)^{s}\binom{m+1}{s}\binom{(n+1)-(m+1)+s}{r}\nonumber\\
			&=&0-\binom{m+1}{0}\binom{(n+1)-(m+1)}{r}=-\binom{n-m}{r}.
		\end{eqnarray}
		Furthermore, if we use
		$q=k-1,\,n=m+k$, and $r=s+k-1$,
		in \eqref{8_tilde}, we obtain
		$$
		\binom{m+1}{s}=\sum_{j=0}^{k-1}\binom{k-1}{j}\binom{m+j+1}{s+k-1}(-1)^{j+k-1}.
		$$
		Replacing it in the last term of \eqref{once},
		\begin{eqnarray}
			\label{12_tilde}
			&&\displaystyle\sum_{s=1}^{m+1}\frac{(-1)^{s}}{s+k}\binom{m+1}{s}\binom{n-m+s}{r}=\nonumber \\
			&=&\sum_{s=1}^{m+1}\binom{n-m+s}{r}\left[\sum_{j=0}^{k-1}\binom{k-1}{j}\binom{m+j+1}{s+k-1}\frac{(-1)^{s+j+k-1}}{s+k}\right]\nonumber\\
			&=&\sum_{j=0}^{k-1}\binom{k-1}{j}\frac{(-1)^{j+k-1}}{m+j+2}\left[\sum_{s=1}^{m+1}\binom{m+j+2}{k+s}\binom{n-m+s}{r}(-1)^{s}\right],
		\end{eqnarray}
		where, from Lemma \ref{nuevo_lema2}, the inner sum is
		\begin{eqnarray}
			\label{trece}
			&&\sum_{s=1}^{m+1}\binom{m+j+2}{k+s}\binom{n-m+s}{r}(-1)^{s}=\nonumber\\
			&=&(-1)^k\sum_{l=k+1}^{m+j+2}\binom{m+j+2}{l}\binom{n-k-m+l}{r}(-1)^{l}\nonumber\\
			&=&\sum_{l=0}^{k}\binom{m+j+2}{l}\binom{n-k-m+l}{r}(-1)^{k+l+1}.
		\end{eqnarray}
		Replacing \eqref{trece} in \eqref{12_tilde} and the result of this, along with \eqref{doce}, in \eqref{once}, we arrive to
		\begin{eqnarray}
			\label{catorce}
			&&\sum_{s=1}^{m+1}\frac{s}{s+k}(-1)^{s}\binom{m+1}{s}\binom{n-m+s}{r}=\\
			&=&-\binom{n-m}{r}-k
			\sum_{j=0}^{k-1}\frac{(-1)^{j+k-1}}{m+j+2}\binom{k-1}{j}\left[\sum_{s=1}^{m+1}(-1)^{s}\binom{m+j+2}{k+s}\binom{n-m+s}{r}\right]\nonumber\\
			&=&-\binom{n-m}{r}-k
			\sum_{s=0}^{k}(-1)^{s}\binom{n-k-m+s}{r}\left[\sum_{j=0}^{k-1}\frac{(-1)^{j}}{m+j+2}\binom{k-1}{j}\binom{m+j+2}{s}\right],\nonumber
		\end{eqnarray}
		where, for 
		$s>0$,
		from Lemma \ref{nuevo_lema2} (with
		$q=k-1,\,r=s-1$ 
		and
		$n=N+k$),
		we have
		\begin{equation}
			\label{quince}
			\sum_{j=0}^{k-1}\frac{(-1)^{j}}{m+j+2}\binom{k-1}{j}\binom{m+j+2}{s}=\frac{1}{s}\sum_{j=0}^{k-1}(-1)^{j}\binom{k-1}{j}\binom{m+j+1}{s-1}=0
		\end{equation}
		because
		$
		j<k.
		$
		Moreover, again from Lemma \ref{nuevo_lema2}, for
		$s=k$
		we see
		\begin{equation}
			\label{dieciseis}
			\sum_{j=0}^{k-1}\frac{(-1)^{j}}{m+j+2}\binom{k-1}{j}\binom{m+j+2}{s}=\frac{1}{k}\sum_{j=0}^{k-1}(-1)^{j}\binom{k-1}{j}\binom{m+j+1}{k-1}=\frac{(-1)^{k-1}}{k}.
		\end{equation}
		Then, taking into account \eqref{quince} and \eqref{dieciseis}, in \eqref{catorce} we obtain
		\begin{equation}
			\label{diecisiete}
			\sum_{s=1}^{m+1}\frac{s}{s+k}(-1)^{s}\binom{m+1}{s}\binom{n-m+s}{r}=-k\binom{n-k-m}{r}\sum_{j=0}^{k-1}\frac{(-1)^{j}}{m+j+2}\binom{k-1}{j}.
		\end{equation}
		Finally, using \eqref{kym} for $m=m+1$ in \eqref{diecisiete}, we arrive to \eqref{10_tilde}.
		
	\end{proof}

	Once we have introduced these auxiliary results, we can finally prove the result stated in the introduction, namely, Theorem \ref{Lema2}.
	
	\noindent
	{\it Proof of Theorem \ref{Lema2}. }
	We proceed by induction on $n$.
	
	For $n=0$ we have $a_{0,0}=\delta_0^{(0)}$ in \eqref{deltas_aes} and \eqref{BB}.
	
	Suppose \eqref{deltas_aes} holds for $n=0,1,\dots$, and assume that \eqref{BB} is satisfied for $n=0,1,\dots,m$, with $m\in\mathbb{N}$. Taking $n=m+1$ in \eqref{deltas_aes}, we have
	$$\delta_{m+1}^{(k)}=\sum_{i=k}^{m+1}\binom{m+1}{i}i!a_{i,i-k}=(m+1)!a_{m+1,m-k+1}+\sum_{i=k}^{m}\binom{m+1}{i}i!a_{i,i-k}$$
	for $k=0,1,\dots,m+1$. Therefore,
	\begin{align}
		\label{eq-aux-1}\nonumber
		(m+1)!a_{m+1,m-k+1} & = \delta_{m+1}^{(k)}-\sum_{i=k}^m\binom{m+1}{i}\left[\sum_{s=k}^i\binom{i}{s}(-1)^{i-s}\delta_s^{(k)}\right]\\ 
		& = \delta_{m+1}^{(k)}-\sum_{s=k}^m\left[\sum_{i=s}^m\binom{m+1}{i}\binom{i}{s}(-1)^{i-s}\right]\delta_s^{(k)}.
	\end{align}
	Applying \eqref{1_lema1},
	\begin{equation}
		\label{eq-aux-2}
		\sum_{i=s}^m\binom{m+1}{s}\binom{i}{s}(-1)^{i-s}=-\binom{m+1}{s}(-1)^{m+1-s}.
	\end{equation}
	Then, from \eqref{eq-aux-1},
	$$(m+1)!a_{m+1,m-k+1}=\delta_{m+1}^{(k)}+\sum_{s=k}^m\binom{m+1}{s}(-1)^{m+1-s}\delta_s^{(k)},$$
	which leads to \eqref{BB} for $n=m+1$.
	
	Conversely, assume \eqref{BB} is satisfied for $n=0,1,\dots,$ and \eqref{deltas_aes} holds for $n=0,1,\dots,m$. Then, taking $n=m+1$ in \eqref{BB},
	\begin{align*}
		\delta_{m+1}^{(k)}  = & (m+1)!a_{m+1,m-k+1}-\sum_{i=k}^m\binom{m+1}{i}(-1)^{m+1-i}\delta_i^{(k)}\\
		= & (m+1)!a_{m+1,m-k+1}-\sum_{i=k}^m\binom{m+1}{i}(-1)^{m+1-i}\left[\sum_{s=k}^i\binom{i}{s}s!a_{s,s-k}\right]\\
		= & (m+1)!a_{m+1,m-k+1}-\sum_{s=k}^m\left[\sum_{i=s}^m\binom{m+1}{i}\binom{i}{s}(-1)^{m+1-i}\right]s!a_{s,s-k}\\
		= & (m+1)!a_{m+1,m-k+1}+\sum_{s=k}^m(-1)^{m-s}\left[\sum_{i=s}^m\binom{m+1}{i}\binom{i}{s}(-1)^{i-s}\right]s!a_{s,s-k}
	\end{align*}
	Using \eqref{eq-aux-2}
	$$\delta_{m+1}^{(k)}=(m+1)!a_{m+1,m-k+1}+\sum_{s=k}^m\binom{m+1}{s}s!a_{s,s-k},$$
	which gives \eqref{deltas_aes}.
	\hfill $\square$
	
	\section{Direct problem}\label{directo}
	
	For this section, we assume that the differential operator \eqref{1} of order $N$ is given; we also assume that there exist both the sequence of eigenvalues 
	$\{\lambda_n\}$
	and the corresponding sequence of eigenpolynomials
	$
	\{P_n\} 
	$
	satisfying \eqref{11} . 
	
	Here and in the sequel, we assume $a_0\equiv 0$ and $\lambda_0=0$ because, otherwise, we would substitute $a_0$ by $a_0-\lambda_0$. Also, we define 
	\begin{equation}
		a_n(x)=0,\quad n>N.
		\label{aaa}
	\end{equation}
	Then, for each 
	$n\in \mathbb{N}$
	we write the polynomials $a_n(x)$ as
	\begin{equation}
		a_n(x)=\sum_{i=0}^n a_{n,i}x^i,\quad a_{n,i}\in \mathbb{C},\quad i=0,1,\ldots, n,
		\label{1111}
	\end{equation}
	and the polynomial $P_n(x)$ as
	\begin{equation}
		P_n(x)=\sum_{i=0}^n b_{n,i}x^i,\quad b_{n,n}=1,\quad b_{n,i}\in \mathbb{C},\quad i=0,1,\ldots, n-1.
		\label{111}
	\end{equation}
	
	Last, we define the sequence $\{\delta_n^{(k)}\}$ as in \eqref{deltas_aes}, understanding 
	$\delta_n^{(k)}=0$
	when 
	$k>n.$
	We remark that, from \eqref{deltas_aes} and \eqref{aaa}, we have 
	$\delta_n^{(k)}=0$
	also when 
	$N+1\leq k\leq n$.
	
	The following result, proved in \cite{AngBar}, provides the relationship between the sequence $\{\delta_n^{(k)}\}$ and the coefficients of the eigenpolynomials $P_n(x)$.
	
	\begin{teo}\cite[Theorem 1]{AngBar}
		\label{lema1} 
		Let 
		$
		\{P_n(x)\}
		$
		and  
		$
		\{\lambda_n\}
		$
		be the sequences of eigenpolynomials and eigenvalues of $L$, respectively, and let $\{\delta_n^{(k)}\}$ be the sequence defined in \eqref{deltas_aes}. Then, for each $n\in \mathbb{N}$, \eqref{11} is equivalent to
		\begin{equation}
			\sum_{k=0}^N\delta_{m+k}^{(k)} b_{n,m+k}=\lambda_nb_{n,m},\quad m=0,1,\ldots, n,
			\label{7}
		\end{equation}
		where $b_{n,m+k},\,k=0,\ldots, ,N,$ are the coefficients of the polynomials $\{P_n\}$ as described in \eqref{111}.
	\end{teo}
	In the statement of Theorem \ref{lema1} and in the rest of the paper, we understand
	$
	b_{n,s}=0
	$
	when 
	$s>n$.
	
	As a consequence of the above result, the eigenvalues of $L$ are determined as 
	\begin{equation}
		\lambda_n=\delta_n^{(0)},\quad n\in \mathbb{N}.
		\label{****}
	\end{equation}
	
	We assume that all the eigenvalues $\lambda_n,\,n=0,1,\ldots,  $ are different from zero. 
	From this, \eqref{deltas_aes} and \eqref{****} we see that 
	$
	a_{n,n}\neq 0$
	in \eqref{1111}
	for some 
	$n\in\{1,2,\ldots, N\}$
	.
	
	Furthermore, we assume that the eigenvalues 
	$\lambda_n,\,n=0,1,\ldots,  $ are also different from each other, which again only depends on the leading coefficients of
	$
	a_n(x)
	$
	in \eqref{1111} (see also  \eqref{deltas_aes}).

	In \cite[Theorem 2]{AngBar}, we provided an expression of the coefficients of the eigenpolynomials $P_n(x)$ in terms of the elements of the sequence $\{\delta_n^{(k)}\}$. Now, we provide a different expression that we find easier to use since it consists of the determinant of a Hessenberg matrix.

	\begin{teo}
		\label{teorema_coeficientes}
		Under the above conditions and using the notation from \eqref{111}, for each $n\in\mathbb{N}$ 
		and
		$1\leq i\leq n$
		we have 
		\begin{equation}
			\label{coeficientes}
			b_{n,n-i}=\left|\begin{array}{ccccccc}
				\displaystyle\frac{\delta_n^{(1)}}{\lambda_n-\lambda_{n-1}} & \displaystyle\frac{\delta_n^{(2)}}{\lambda_n-\lambda_{n-2}} & \cdots & \displaystyle\frac{\delta_n^{(N)}}{\lambda_n-\lambda_{n-N}} & 0 & \textcolor{white}{\displaystyle\frac{\delta_n^{(1)}}{\lambda_n-\lambda_{n-1}}} &\\[5mm]
				-1 & \displaystyle\frac{\delta_{n-1}^{(1)}}{\lambda_n-\lambda_{n-2}} & \cdots & \displaystyle\frac{\delta_{n-1}^{(N-1)}}{\lambda_n-\lambda_{n-N}} & \displaystyle\frac{\delta_{n-1}^{(N)}}{\lambda_n-\lambda_{n-N+1}} & \ddots & \\[5mm]
				0 & -1 & \ddots & & \ddots & \ddots & 0\\[5mm]
				& \ddots & \ddots & \ddots & & \ddots & \displaystyle\frac{\delta_{n-i+N}^{(N)}}{\lambda_n-\lambda_{n-i}}\\[5mm]
				& & \ddots & \ddots & \ddots & & \displaystyle\frac{\delta_{n-i+N-1}^{(N-1)}}{\lambda_n-\lambda_{n-i}}\\[5mm]
				& & & \ddots & \ddots & \ddots & \vdots\\[5mm]
				& & & & 0 & -1 & \displaystyle\frac{\delta_{n-i+1}^{(1)}}{\lambda_n-\lambda_{n-i}}
			\end{array}\right|
		\end{equation}
		where the determinant is correspondingly truncated when $n\leq N$.
	\end{teo}
	
	\begin{proof}
		Let us prove the theorem by induction on $i$.
		
		From \eqref{7}, we obtain
		\begin{equation}
			\label{eq-5}
			b_{n,n-i}=\displaystyle\frac{\displaystyle\sum_{k=1}^N\delta_{n-i+k}^{(k)}b_{n,n-i+k}}{\lambda_n-\lambda_{n-i}},\quad i=1,2,\ldots,n.
		\end{equation}
		If $i=1$, \eqref{eq-5} leads to
		$$b_{n,n-1}=\frac{\delta_n^{(1)}}{\lambda_n-\lambda_{n-1}},$$
		then \eqref{coeficientes} holds for $i=1$.
		
		Now, assume \eqref{coeficientes} holds up to $i-1$ and let us analyze the case $i$.
		
		\noindent We must notice that $b_{n,n-i+k}=0$ if $k> i$. Then, if $i\leq N$, the upper bound of the summation in \eqref{eq-5} can be restricted to $i$. Moreover, we recall that 
		$\delta_{n-i+k}^{(k)}=0$ for $k\geq N+1$. Therefore, either with $i\leq N$ or with $i>N$, we can write \eqref{eq-5} as
		\begin{eqnarray}
			\label{eq-7}
			b_{n,n-i}&=&\frac{\delta_{n-i+1}^{(1)}}{\lambda_n-\lambda_{n-i}}b_{n,n-i+1}+\sum_{k=2}^i\frac{\delta_{n-i+k}^{(k)}}{\lambda_n-\lambda_{n-i}}b_{n,n-i+k}\nonumber\\
			&=&\frac{\delta_{n-i+1}^{(1)}}{\lambda_n-\lambda_{n-i}}b_{n,n-i+1}+\sum_{j=1}^{i-1}\frac{\delta_{n-j+1}^{(i-j+1)}}{\lambda_n-\lambda_{n-i}}b_{n,n-j+1}.
		\end{eqnarray}
		
		On the other hand, it is easy to see that, for any upper Hessenberg matrix 
		$H$ 
		of order 
		$m$ 
		with entries
		$h_{ij}, i,j=1,2,\ldots,m,$
		and leading principal submatrices $H_s$ of order 
		$s=1,2,\ldots,m-1,$
		we have
		\begin{equation}
			\det(H)=h_{m,m}\det(H_{m-1})+\sum_{j=1}^{m-1}(-1)^{m-j}h_{j,m}\prod_{\ell=j}^{m-1}h_{\ell+1,\ell}\det(H_{j-1}).
			\label{LemaHes}
		\end{equation}
		Thanks to the induction hypothesis, we know $b_{n,n-i+1}$ and $b_{n,n+1-j}, 1\leq j\leq i-1$ are the determinants defined in \eqref{coeficientes}, that is, they are determinants of matrices of size $i-1$ and $j-1$, respectively. Then, comparing \eqref{eq-7} with \eqref{LemaHes} we obtain that $b_{n,n-i}$ is the determinant of an upper Hessenberg matrix $H$ of order $i$ with
		$$h_{i,i}=\frac{\delta_{n-i+1}^{(1)}}{\lambda_n-\lambda_{n-i}}, \quad \det(H_{i-1})=b_{n,n-i+1}, \quad h_{j,i}=\frac{\delta_{n+1-j}^{(i+1-j)}}{\lambda_n-\lambda_{n-i}}, \quad \det(H_{j-1})=b_{n,n-j+1}.$$
		It is easy to check $H$ is the matrix whose determinant appears in \eqref{coeficientes}, which concludes our proof.
			
	\end{proof}
	
	In the following theorem, the expression of $\delta_n^{(k)},\, n>N$, in terms of $\delta_1^{(k)},\dots,\delta_N^{(k)}$ is given.
	
	\begin{teo}
		\label{th-main}
		Let $L$ be the operator defined in \eqref{1} and let $\{\delta_n^{(k)}\}$ be the sequence defined in \eqref{deltas_aes}. Then,
		\begin{equation}
			\label{eq-2}
			\begin{cases}
				\displaystyle\delta_n^{(k)}=\sum_{i=k}^{N}(-1)^{N-i}\binom{n}{i}\binom{n-i-1}{N-i}\delta_i^{(k)},\quad n\geq N+1.\\
				\delta_n^{(k)}=0,\quad n\geq k>N.
			\end{cases}
		\end{equation}
	\end{teo}
	
	\begin{proof} 
		
		Let us prove \eqref{eq-2} by induction on $n$. The base case is $n=N+1$.
		Since the sequence \eqref{deltas_aes} is equivalent to \eqref{BB} (see Theorem \ref{lema1}) and \eqref{aaa}, we obtain $a_{N+1}\equiv 0$ and, for $k\leq N+1$,
		$$
		0=(N+1)!a_{N+1,N+1-k}=\sum_{i=k}^{N+1}\binom{N+1}{i}(-1)^{N+1-i}\delta_i^{(k)}.
		$$
		Therefore,
		$$
		\delta_{N+1}^{(k)}=\sum_{i=k}^N\binom{N+1}{i}(-1)^{N-i}\delta_i^{(k)}.
		$$
		Notice that this is exactly \eqref{eq-2} for $n=N+1$, so the base case is proved.
		
		For the inductive step, assume the result is true for $n\geq N+1$. For $n+1$, \eqref{BB} is 
		$$
		0=(n+1)!a_{n+1,n+1-k}=\sum_{i=k}^{n+1}\binom{n+1}{i}(-1)^{n+1-i}\delta_i^{(k)},
		$$
		which is equivalent to
		$$\delta_{n+1}^{(k)}=\sum_{i=k}^{n}\binom{n+1}{i}(-1)^{n-i}\delta_i^{(k)}.
		$$
		
		Applying the induction hypothesis,
		\begin{equation}
			\label{eq-ind-1}
			\delta_{n+1}^{(k)}=\sum_{i=k}^{N}\binom{n+1}{i}(-1)^{n-i}\delta_i^{(k)}+\sum_{i=N+1}^n\binom{n+1}{i}(-1)^{n-i}\sum_{\ell=k}^{N}(-1)^{N-\ell}\binom{i}{\ell}\binom{i-\ell-1}{N-\ell}\delta_\ell^{(k)}.
		\end{equation}
		In \eqref{eq-ind-1}, the coefficient that multiplies each $\delta_\ell^{(k)}$ for 
		$\ell\geq k$ is
		\begin{equation}
			\label{eq-aux1}
			\binom{n+1}{\ell}(-1)^{n-\ell}+\sum_{i=N+1}^n\binom{n+1}{i}(-1)^{N-\ell+n-i}\binom{i}{\ell}\binom{i-\ell-1}{N-\ell}.
		\end{equation}
		In order to prove our result, we must show the above coefficient is equal to the coefficient that multiplies each $\delta_\ell^{(k)}$ in the expression of $\delta_{n+1}^{(k)}$ given by \eqref{eq-2}. This coefficient is 
		\begin{equation}
			\label{eq-aux2}
			(-1)^{N-\ell}\binom{n+1}{\ell}\binom{n-\ell}{N-\ell}=(-1)^{N-\ell}\frac{(n+1)!}{\ell!(N-\ell)!(n-N)!(n+1-\ell)}.
		\end{equation}
		
		To prove \eqref{eq-aux1} coincides with the right hand side of \eqref{eq-aux2} let us expand \eqref{eq-aux1}, which is
		\begin{align}
			\label{eq-x}
			&\frac{(n+1)!}{(n+1-\ell)!\ell!}(-1)^{n-\ell}+\sum_{i=N+1}^n\frac{(n+1)!}{(n+1-i)!i!}\frac{(-1)^{n-i}(-1)^{N-\ell}}{\ell!(N-\ell)!}\frac{i!}{(i-N-1)!(i-\ell)}\nonumber\\ 
			=&(-1)^{N-\ell}\frac{(n+1)!}{\ell!(N-\ell)!}\left(\frac{(-1)^{n-N}}{(n+1-\ell)(n-\ell)\cdots(N+1-\ell)}+\sum_{i=N+1}^n\frac{(-1)^{n-i}}{(n+1-i)!(i-N-1)!(i-\ell)}\right).
		\end{align}
		Taking $s=\ell-N$ and $j=i-N$,  in \eqref{eq-x} we can write  
		\begin{equation}
			\label{aux}
			(-1)^{N-\ell}\frac{(n+1)!}{\ell!(N-\ell)!}\left(\frac{(-1)^{n-N}}{(1-s)(2-s)\cdots(n-N+1-s)}+\sum_{j=1}^{n-N}\frac{(-1)^{n-N-j}}{(n-N+1-j)!(j-1)!(j-s)}\right).
		\end{equation}
		
		Now we will prove the following partial fraction decomposition for 
		$
		n\in \mathbb{N},\, n\geq N,\, s\in\mathbb{C},\, s\neq 1,2,\ldots,  n-N+1$
		,
		\begin{equation}
			\label{31tilde}
			\frac{1}{(1-s)(2-s)\cdots(n-N+1-s)}=\sum_{i=1}^{n-N+1}\frac{(-1)^{i-1}}{(n-N+1-i)!(i-1)!(i-s)}.
		\end{equation}
		In fact, we would like to obtain the coefficients for the partial fraction decomposition
		$$\frac{1}{(1-s)(2-s)\cdots(n-N+1-s)}=\frac{A_1}{1-s}+\frac{A_2}{2-s}+\cdots\frac{A_{n-N+1}}{n-N+1-s}.$$
		The right hand side in the above expression can be written with a common denominator as
		$$\sum_{i=1}^{n-N+1}\frac{A_i}{i-s}=
		\sum_{i=1}^{n-N+1} \frac{A_i\prod_{k\neq i} (k-s)}{\prod_{k=1}^{n-N+1} (k-s)}.$$
		Consequently,
		$$
		1=\sum_{i=1}^{n-N+1} A_i\prod_{k\neq i} (k-s)
		$$
		and taking $s=i, 1\leq i\leq n-N+1$, we arrive to \eqref{31tilde}.
		
		Using the above partial fraction decomposition in \eqref{aux}, we obtain
		\begin{align*}
			&(-1)^{n-\ell}\frac{(n+1)!}{\ell!(N-\ell)!}\left(\sum_{i=1}^{n-N+1}\frac{(-1)^{i-1}}{(n-N+1-i)!(i-1)!(i-s)}+\sum_{j=1}^{n-N}\frac{(-1)^j}{(n-N+1-j)!(j-1)!(j-s)}\right)\\
			&=(-1)^{n-\ell}\frac{(n+1)!}{\ell!(N-\ell)!}\frac{(-1)^{n-N}}{(n-N)!(n-N+1-s)}=(-1)^{N-\ell}\frac{(n+1)!}{\ell!(N-\ell)!}\frac{1}{(n-N)!(n-N+1-s)}.
		\end{align*}
		Undoing the change $s=\ell-N$, this is
		$$(-1)^{N-\ell}\frac{(n+1)!}{\ell!(N-\ell)!(n-N)!(n+1-\ell)},$$
		which is exactly \eqref{eq-aux2}.	
		
		Finally, combining \eqref{aaa} and \eqref{deltas_aes} we obtain $\delta_n^{(k)}=0,\ n\geq k>N$, as we remarked after \eqref{111}.
		
	\end{proof}
	
	Combining theorems \ref{teorema_coeficientes} and \ref{th-main} we obtain an expression of each coefficient $b_{n,n-i}$ that depends, at most, on $\lambda_1, \dots, \lambda_N, \delta_1^{(1)},\dots,\delta_N^{(1)},\delta_2^{(2)},\dots,\delta_N^{(2)},$ $\delta_3^{(3)},\dots,\delta_N^{(3)},\dots,\delta_N^{(N)}$. In particular, if $n-i\geq N$, the expression for the coefficient is very simple, as shown in the following result.
	
	\begin{coro}
		Let $n-i\geq N$. Then, we have
		\begin{equation}
			\label{eq-b}
			b_{n,n-i}=\frac{n(n-1)\cdots(n-i+1)}{\prod_{\ell=1}^{i}\left(\sum_{m=1}^N(-1)^{N-m}\binom{N}{m}\left(\prod_{\substack{j=0\\ j\neq m}}^{N}(n-j)-\prod_{\substack{j=0\\ j\neq m}}^{N}(n-\ell-j)\right)\lambda_m\right)}\det(M),
		\end{equation}
		where $M$ is an upper Hessenberg $i\times i$ matrix that has
		\begin{itemize}
			\item 	$\displaystyle\frac{1}{n-\ell}\sum_{i=1}^N(-1)^{N-i}\binom{N}{i}\left(\prod_{\substack{j=0\\ j\neq i}}^{N}(n-\ell-j)-\prod_{\substack{j=0\\ j\neq i}}^{N}(n-j)\right)\lambda_i$ as the $\ell$-th entry of the subdiagonal.
			
			\item $\displaystyle\sum_{i=1}^N(-1)^{N-i}\binom{N}{i}\prod_{j=1,j\neq i}^{N}(n-\ell+1-j)\delta_i^{(1)}$ as the $\ell$-th entry of the main diagonal.
			\item $\displaystyle\sum_{i=k}^N(-1)^{N-i}\binom{N}{i}\prod_{j=1,j\neq i}^{N}(n-\ell+1-j)\delta_i^{(k)}$ as the $\ell$-th entry of the $(k-1)$-superdiagonal, $\,\,k=2,\ldots, N-1$.
			\item $\displaystyle\prod_{j=1}^{N-1}(n-\ell+1-j)\delta_N^{(N)}$ as the $\ell$-th entry of the $(N-1)$-superdiagonal.
			\item $[0,\cdots,0]$ as the $m$-superdiagonal for $m\geq N$.
		\end{itemize}
	\end{coro}
	
	\begin{proof}
		From \eqref{coeficientes} in Theorem \ref{teorema_coeficientes}, we have	
		\begin{align*}
			b_{n,n-i}= \frac{1}{(\lambda_n-\lambda_{n-1})\cdots(\lambda_n-\lambda_{n-i})}\left|\begin{array}{ccccccc}
				\delta_n^{(1)} & \delta_n^{(2)} & \cdots & \delta_n^{(N)} & 0 & \textcolor{white}{\displaystyle\frac{\delta_n^{(1)}}{\lambda_n-\lambda_{n-1}}} &\\[5mm]
				\lambda_{n-1}-\lambda_n & \delta_{n-1}^{(1)} & \cdots & \delta_{n-1}^{(N-1)}& \delta_{n-1}^{(N)} & \ddots & \\[5mm]
				0 & \lambda_{n-2}-\lambda_n & \ddots & & \ddots & \ddots & 0\\[5mm]
				& \ddots & \ddots & \ddots & & \ddots & \delta_{n-i+N}^{(N)}\\[5mm]
				& & \ddots & \ddots & \ddots & & \delta_{n-i+N-1}^{(N-1)}\\[5mm]
				& & & \ddots & \ddots & \ddots & \vdots\\[5mm]
				& & & & 0 & \lambda_{n-i+1}-\lambda_n & \delta_{n-i+1}^{(1)}
			\end{array}\right|.
		\end{align*}	
		Notice that the $m$-th row of the determinant contains entries of the form $\delta_{n-m+1}^{(k)}$. Since $n-i\geq N$, we can apply \eqref{eq-2} to each row and extract $n-m+1$ from each row, leading to \eqref{eq-b}.
		
	\end{proof}
	
	\begin{rem}
		Taking $k=0$ in \eqref{eq-2}, we underline that the sequence of eigenvalues $\{\lambda_n\}$ of every differential operator \eqref{1} satisfies
		$$\displaystyle\lambda_n=\sum_{i=1}^{N}(-1)^{N-i}\binom{n}{i}\binom{n-i-1}{N-i}\lambda_i, \quad n\geq N+1.$$
		In particular, if $L$ in \eqref{1} is an operator of order $N=2$, then
		\begin{equation}
			\label{eq-lambda2}
			\lambda_n=-n(n-2)\lambda_1+\frac{n(n-1)}{2}\lambda_2.
		\end{equation}
		That is, all the eigenvalues of the classical orthogonal polynomials are a linear combination of $\lambda_1$ and $\lambda_2$ with the coefficients given by \eqref{eq-lambda2}. Notice that if we recall the general expression of the eigenvalues of the classical orthogonal polynomials, it is trivial to verify \eqref{eq-lambda2}:
		\begin{itemize}
			\item For Jacobi polynomials, $\lambda_n=-n(n+1+\alpha+\beta),\ \alpha,\beta>-1$.
			\item For Laguerre polynomials, $\lambda_n=-n$.
			\item For Hermite polynomials, $\lambda_n=-2n$.
		\end{itemize}
		However, this fact seems to have remained unnoticed in the literature to the best of our knowledge.
	\end{rem}
	
	\section{Example of use of $\{\delta_n^{(k)}\}$ for a particular Bochner differential operator}
	\label{ejemplo}
	
	In this section we study the Bochner differential operator
	\begin{equation}
		\label{Delta}
		L:=\sum_{i=1}^Nc_i x^{i-1}\partial_x^{i}+x\partial_x.
	\end{equation}
	
	The reason to study this operator in particular is Conjecture 1.10 in \cite{Shapiro2}. This conjecture claims that the only two $N$-degree Bochner differential operators such that their eigenpolynomials satisfy an $(N+1)$-term recurrence are \eqref{Delta} and 
	\begin{equation}
		\label{operador_otro}
		L=q'(G)G+x\delta_x,
	\end{equation}
	where, given $\ell$ a divisor of $N$, 
	$$G := \sum_{m=0}^{\ell-1}a_m(x\partial_x)^m\partial_x$$
	and $q(t)$ is any complex polynomial of degree $N/\ell$ without a constant term.
	
	That is, if this conjecture is true, only these two Bochner differential operators lead to solutions of this bispectral problem.
	
	The case $N=3$ is studied and proved to be true in \cite{Shapiro2}. However, the techniques used in that work are extremely complicated and it seems hard to extend them for general values of $N$.  Consequently, we propose to address the problem using the sequences 
	$
	\{
	\delta_n^{(k)}
	\},\,n\in \mathbb{N},
	\,k=0,1,\ldots$,
	given in \eqref{deltas_aes}. In this paper, we will only focus in using the sequence $
	\{
	\delta_n^{(k)}
	\}$ to prove that the eigenpolynomials of \eqref{Delta} satisfy an $(N+1)$-term recurrence relation. We hope to address the remaining part of the conjecture in a future work.
	
	Before proving our result, let us make some preliminary simplifications. The polynomials $a_i(x)$ defining the operator \eqref{Delta} are
	$$
	a_i(x)=\left\{
	\begin{array}{ll}
		x+c_1,&\text{if  }i=1,\\
		c_ix^{i-1},&\text{if  }i=2,\dots,N.
	\end{array}
	\right.
	$$ 
	Hence, using \eqref{deltas_aes}, we can deduce
	\begin{equation}
		\label{deltas}
		\delta_n^{(k)}=
		\left\{
		\begin{array}{clll}
			n, & \text{if }& k=0,\\
			\displaystyle\sum_{s=k}^n \binom n s s!c_s,& \text{if }& k=1,\\
			0 ,& \text{if }& k>1.\\
		\end{array}
		\right.
	\end{equation}
	In particular (see \eqref{****}), the eigenvalues are 
	$\lambda_n=n.$
	
	With the notation of \eqref{111}, from \eqref{deltas} and Theorem \ref{teorema_coeficientes},
	
	$$
	(\lambda_n-\lambda_{n-1})(\lambda_n-\lambda_{n-2})\cdots (\lambda_n-\lambda_{n-i})b_{n,n-i}=
	\left|\begin{array}{ccccccc}
		\displaystyle\delta_n^{(1)}& 0 & \cdots & \cdots & \cdots& \cdots & 0\\
		-1 &\delta_{n-1}^{(1)} & 0& \cdots & \cdots& \cdots & 0\\
		0 & -2 & \ddots & \ddots && & \vdots\\
		\vdots & \ddots & \ddots & \ddots &  \ddots &&\vdots\\
		\vdots & & & \ddots & \ddots & \ddots & \vdots\\
		0& \cdots &\cdots & \cdots &- i+2 & \delta_{n-i+2}^{(1)}& 0\\
		0& \cdots &\cdots & \cdots& 0 &- i+1 & \delta_{n-i+1}^{(1)}
	\end{array}\right|.
	$$
	This is,
	\begin{equation}
		\label{0}
		b_{n,n-i}=\frac{\delta_{n}^{(1)} \delta_{n-1}^{(1)} \ldots \delta_{n-i+1}^{(1)} }{i!}
	\end{equation}
	or, what is the same,
	$$
	P_n(x)=x^n+\sum_{i=1}^{n}\frac{\delta_{n}^{(1)} \delta_{n-1}^{(1)} \ldots \delta_{n-i+1}^{(1)} }{i!}x^{n-i},\quad n\in \mathbb{N}.
	$$
	
	Now we can prove our result. That is, we prove that the above sequence of eigenpolynomials of $L$ follows an $(N+1)$-term recurrence relation.
	
	\begin{teo} Let $L$ be the differential operator introduced in \eqref{Delta}. If 
		$\{P_n\}$
		is the sequence of eigenpolynomials and
		$\{\lambda_n\},\,n\in \mathbb{N},$
		the corresponding sequence of eigenvalues for $L$,
		(that is,
		$LP_n=\lambda_nP_n,\,n\in\mathbb{N},$)
		then the following recurrence relation is satisfied,
		\begin{equation}
			\label{recu}
			\left.
			\begin{array}{rlll}
				\displaystyle\sum_{s=1}^{N-1}\alpha_{n,n-s}P_{n-s}(x)+(\alpha_{n,n}-x)P_n(x)+P_{n+1}(x)=0,& n=0,1,\ldots\\
				P_{-N+1}=P_{-N+2}=\cdots =P_{-1}=0,&P_0= 1,
			\end{array}
			\right\}
		\end{equation}
		with
		\begin{equation}
			\label{*}
			\alpha_{n,n-s}=\frac{\delta^{(1)}_n\delta^{(1)}_{n-1}\ldots \delta^{(1)}_{n-s+1}}{(s+1)!}
			\sum_{j=0}^{s+1}\binom {s+1} j(-1)^j\delta^{(1)}_{n-s+j},\quad s=0,\ldots, N-1,
		\end{equation}
		(understanding 
		$\delta^{(1)}_n\delta^{(1)}_{n-1}\ldots \delta^{(1)}_{n-s+1}=1$
		when $s=0$).
		
	\end{teo}
	\begin{proof}
		With the above notation, it is immediate to see that \eqref{recu} can be written as
		\begin{equation}\label{recurrencia}
			\left.
			\begin{array}{rlll}
				\displaystyle\sum_{s=0}^{N-1}\alpha_{n,n-s}b_{n-s,i}=b_{n,i-1}-b_{n+1,i}\\
				n=0,1,\ldots,\quad i=0,1,\ldots, n+1,\\
			\end{array}
			\right\}
		\end{equation}
		where, as usual,
		$b_{m,j}=0$
		for
		$j>m$ or
		$m<0$.
		
		For each fixed
		$n\in \{0,1,\ldots\}$
		the set of polynomials 
		$
		\{P_0(x),\,P_1(x),\ldots, P_n(x),\,xP_n(x)\}
		$
		constitutes a base of the space of polynomials of degree up to $n+1$. Hence, there exist 
		$
		\alpha_{n,0},\,\alpha_{n,1},\,\ldots \alpha_{n,n}
		$
		such that 
		\begin{equation}
			\label{otra_recu}
			\displaystyle\sum_{s=1}^{n}\alpha_{n,n-s}P_{n-s}(x)+(\alpha_{n,n}-x)P_n(x)+P_{n+1}(x)=0
		\end{equation}
		and
		\begin{equation}
			\label{otras_alfas}
			\displaystyle\sum_{s=0}^{n}\alpha_{n,n-s}b_{n-s,i}=b_{n,i-1}-b_{n+1,i},\quad i=0,1,\ldots, n+1.
		\end{equation}
		As a consequence, \eqref{recu} and \eqref{recurrencia} are satisfied for 
		$
		n=0,1,\ldots, N-1
		$
		(assuming 
		$\alpha_{n,i}=0$
		for 
		$i<0$). 
		Moreover, for each fixed 
		$n\in \{0,1,\ldots\}$, \eqref{otras_alfas} can be interpreted as a triangular system whose solutions, obtained successively taking
		$i=n,\,n-1,\,\ldots, n-N+1,$
		are
		\begin{equation}
			\label{dos}
			\left.
			\begin{array}{ccl}
				\alpha_{n,n}& = & b_{n,n-1}-b_{n+1,n}\\
				\alpha_{n,n-1}& = & -\alpha_{n,n}b_{n,n-1}+b_{n,n-2}-b_{n+1,n-1}\\
				\alpha_{n,n-2}& = & -\alpha_{n,n-1}b_{n-1,n-2}-\alpha_{n,n}b_{n,n-2}+b_{n,n-3}-b_{n+1,n-2}\\
				\vdots & \vdots & \vdots  \\
				\alpha_{n,n-N+1}& = & -\alpha_{n,n-N+2}b_{n-N+2,n-N+1}-\cdots - \alpha_{n,n}b_{n,n-N+1}+b_{n,n-N}-b_{n+1,n-N+1}.\\
			\end{array}
			\right\}
		\end{equation}
		The idea now is to use \eqref{0} and \eqref{dos} to prove \eqref{*}. 
		
		Firstly, notice that \eqref{dos} implies \eqref{*} holds for 
		$s=0$. 
		Secondly, assuming \eqref{*} is satisfied for
		$s=0,1,\ldots, k$
		with 
		$k< N-1$, from \eqref{dos} we have
		\begin{equation}
			\label{tres}
			\alpha_{n,n-k-1}= -\sum_{s=0}^k \alpha_{n,n-s}b_{n-s,n-k-1}+b_{n,n-k-2}-b_{n+1,n-k-1}
		\end{equation}
		where, for 
		$ s=0,1,\ldots,k,$
		\begin{eqnarray}
			\alpha_{n,n-s}b_{n-s,n-k-1}& = & \left(\frac{\delta^{(1)}_n\delta^{(1)}_{n-1}\ldots \delta^{(1)}_{n-s+1}}{ (s+1)!}\right)
			\left(\frac{\delta^{(1)}_{n-s}\delta^{(1)}_{n-s-1}\ldots \delta^{(1)}_{n-k}}{ (k-s+1)!}\right)
			\sum_{j=0}^{s+1}\binom {s+1} j(-1)^j\delta^{(1)}_{n-s+j}\nonumber\\
			& = &  \frac{\delta^{(1)}_n\delta^{(1)}_{n-1}\ldots \delta^{(1)}_{n-k}}{ (k+2)!}
			\sum_{j=0}^{s+1}\frac{(k+2)!(-1)^j}{(k-s+1)!j!(s-j+1)!}\delta^{(1)}_{n-s+j}.
			\label{cuatro}
		\end{eqnarray}
		On the other hand,
		\begin{equation}
			b_{n,n-k-2}-b_{n+1,n-k-1}=
			\frac{\delta^{(1)}_n\delta^{(1)}_{n-1}\ldots \delta^{(1)}_{n-k-1}-\delta^{(1)}_{n+1}\delta^{(1)}_{n}\ldots \delta^{(1)}_{n-k}    }{ (k+2)!}.
			\label{cinco}
		\end{equation}
		
		Substituting \eqref{cuatro} and \eqref{cinco} in \eqref{tres},
		\begin{eqnarray*}
			\alpha_{n,n-k-1}&=&-\sum_{s=0}^k\left( \frac{\delta^{(1)}_n\delta^{(1)}_{n-1}\ldots \delta^{(1)}_{n-k}}{ (k+2)!}
			\sum_{j=0}^{s+1}\frac{(k+2)!(-1)^j}{(k-s+1)!j!(s-j+1)!}\delta^{(1)}_{n-s+j}\right)\\
			&&+\frac{\delta^{(1)}_n\delta^{(1)}_{n-1}\ldots \delta^{(1)}_{n-k}}{ (k+2)!}\left(\delta^{(1)}_{n-k-1}-\delta^{(1)}_{n+1}\right)\\
			&=& \frac{\delta^{(1)}_n\delta^{(1)}_{n-1}\ldots \delta^{(1)}_{n-k}}{ (k+2)!}\left(\sum_{s=0}^k
			\sum_{j=0}^{s+1}\binom{k+2}{s+1}\binom{s+1}{j}(-1)^{j+1}\delta^{(1)}_{n-s+j}+\delta^{(1)}_{n-k-1}-\delta^{(1)}_{n+1}\right)\\
			&=& \frac{\delta^{(1)}_n\delta^{(1)}_{n-1}\ldots \delta^{(1)}_{n-k}}{ (k+2)!}\left(\sum_{r=1}^{k+2}\binom{k+2}{r}
			\delta^{(1)}_{n-k+r-1}\displaystyle\sum_{\ell=1}^{r}\binom{r}{\ell} (-1)^{r-\ell+1}+\delta^{(1)}_{n-k-1}\right)\\
			&=&\frac{\delta^{(1)}_n\delta^{(1)}_{n-1}\ldots \delta^{(1)}_{n-k}}{ (k+2)!}\left(\sum_{r=1}^{k+2}\binom{k+2}{r}(-1)^r
			\delta^{(1)}_{n-k+r-1}+\delta^{(1)}_{n-k-1}\right)\\
			&=& \frac{\delta^{(1)}_n\delta^{(1)}_{n-1}\ldots \delta^{(1)}_{n-k}}{ (k+2)!}\sum_{r=0}^{k+2}\binom{k+2}{r}(-1)^r
			\delta^{(1)}_{n-k+r-1},
		\end{eqnarray*}
		which is \eqref{*} for
		$s=k+1$. With this, \eqref{*} is verified for 
		$k=0,1,\ldots, N-1.$
		
		The next step is to prove \eqref{recurrencia}, from which \eqref{recu} will be a consequence. Due to \eqref{otra_recu}-\eqref{otras_alfas} and the fact that 
		$\alpha_{n,i}=0$
		for 
		$i<0$, 
		the recurrence relation \eqref{recu} is obviously fulfilled for 
		$n\le N-1$.
		Furthermore, for 
		$n\ge N$ and
		$i=n-N+1,\,n-N+2,\,\ldots, n, $
		the relation \eqref{recurrencia} is obvious from \eqref{dos}.  Then, in the rest of the proof we assume 
		$n\geq N,\, i\in \{0,\,1,\,\ldots, n-N\}.$
		That is,
		$i=n-N-k,\,k=0,1,\ldots,n-N.$

		On the left hand side of \eqref{recurrencia}, using \eqref{*} and \eqref{0}, we have
		\begin{eqnarray}
			\label{6_tilde}
			\sum_{s=0}^{N-1}\alpha_{n,n-s}b_{n-s,i}&=&\sum_{s=0}^{N-1}
			\frac{\delta^{(1)}_n\delta^{(1)}_{n-1}\ldots \delta^{(1)}_{n-s+1}}{(s+1)!}
			\left(\sum_{j=0}^{s+1}\binom {s+1} j(-1)^j\delta^{(1)}_{n-s+j}\right) \frac{\delta_{n-s}^{(1)} \delta_{n-s-1}^{(1)} \ldots \delta_{i+1}^{(1)} }{(n-s-i)!}\nonumber\\
			&=&\frac{\delta_{n}^{(1)} \delta_{n-1}^{(1)} \ldots \delta_{i+1}^{(1)} }{(n-i+1)!}\sum_{s=0}^{N-1}\binom {n-i+1}{s+1}\sum_{j=0}^{s+1}\binom {s+1} j(-1)^j\delta^{(1)}_{n-s+j}\nonumber\\
			&=&\frac{\delta_{n}^{(1)} \delta_{n-1}^{(1)} \ldots \delta_{i+1}^{(1)} }{(n-i+1)!}\left(\sum_{s=0}^{N}\sum_{j=s}^{N}\binom {n-i+1}{j}\binom {j} {j-s}(-1)^{j-s}\delta^{(1)}_{n-s+1}-\delta^{(1)}_{n+1}\right)\nonumber\\
			&=&\frac{\delta_{n}^{(1)} \delta_{n-1}^{(1)} \ldots \delta_{i+1}^{(1)} }{(n-i+1)!}\sum_{s=0}^{N}\sum_{j=s}^{N}\binom {n-i+1}{j}\binom {j} {s}(-1)^{j-s}\delta^{(1)}_{n-s+1}-b_{n+1,i}.
		\end{eqnarray}
		Now, it is easy to see
		\begin{eqnarray*}
			\sum_{j=s}^{N}\binom {n-i+1}{j}\binom {j} {s}(-1)^{j-s}&=&\binom {n-i+1}{s}\sum_{j=s}^{N}\binom {n-s-i+1} {j-s}(-1)^{j-s}\\
			&=&\binom {n-i+1}{s}\sum_{r=0}^{N-s}\binom {n-s-i+1} {r}(-1)^{r}\\
			&=&\binom {n-i+1}{s}\binom {n-s-i} {n-N-i}(-1)^{N-s},
		\end{eqnarray*}
		where the last equality follows from \eqref{2_lema1}.
		
		Hence, substituting 
		$i=n-N-k$,
		in \eqref{6_tilde} and combining it with the above results, we obtain
		\begin{eqnarray}
			\label{casi_tilde}
			\sum_{s=0}^{N-1}\alpha_{n,n-s}b_{n-s,i}+b_{n+1,i}&=&\frac{\delta_{n}^{(1)} \delta_{n-1}^{(1)} \ldots \delta_{n-N-k+1}^{(1)} }{(N+k+1)!}
			\sum_{s=0}^{N}\binom {N+k+1}{s}\binom {N+k-s} {k}(-1)^{N-s}\delta^{(1)}_{n-s+1} \\
			&=&\frac{\delta_{n}^{(1)} \delta_{n-1}^{(1)} \ldots \delta_{n-N-k+1}^{(1)} }{(N+k+1)!}\binom {N+k+1}{k}\sum_{s=0}^{N}\frac{N+1}{N+k-s+1}\binom {N} {s}(-1)^{N-s}\delta^{(1)}_{n-s+1}\nonumber\\
			&=&\frac{\delta_{n}^{(1)} \delta_{n-1}^{(1)} \ldots \delta_{n-N-k+1}^{(1)} }{(N+k+1)!}\binom {N+k+1}{k}\sum_{s=1}^{N+1}\frac{s}{s+k}\binom {N+1} {s}(-1)^{s+1}\delta^{(1)}_{n-N+s},\nonumber
		\end{eqnarray}
		where the last equality follows from performing the change of variable $s=N-s+1$. 
		
		Now, if we take into account \eqref{deltas},
		\begin{eqnarray}
			\label{sustituir}
			\sum_{s=1}^{N+1}\frac{s}{s+k}\binom {N+1} {s}(-1)^{s+1}\delta^{(1)}_{n-N+s}&=&
			\sum_{s=1}^{N+1}\frac{s}{s+k}\binom {N+1} {s}(-1)^{s+1}\sum_{j=1}^{n-N+s}\binom{n-N+s}{j}j!c_j\nonumber\\
			&=& \sum_{j=1}^Nj!c_j\sum_{s=1}^{N+1}\frac{s}{s+k}\binom {N+1} {s}\binom{n-N+s}{j}(-1)^{s+1}.
		\end{eqnarray}
		(We note that
		$c_{N+1}=0$ and $n\geq N$
		are used in \eqref{sustituir}.)
		
		Applying Lemma \ref{lema3} in \eqref{sustituir} (with $m=N$),
		\begin{equation}
			\label{casi}
			\binom {N+k+1}{k}\sum_{s=1}^{N+1}\frac{s}{s+k}\binom {N+1} {s}(-1)^{s+1}\delta^{(1)}_{n-N+s}= \sum_{j=1}^N
			\binom{n-N-k}{j}
			j!c_j.
		\end{equation}
		In \eqref{casi} we have two possibilities. First, if 
		$n-N-k\le N$
		then
		$
		\displaystyle\binom{n-N-k}{j}=0
		$
		for 
		$j=n-N-k+1,\ldots,N.$
		Second, if 
		$n-N-k> N$
		then
		$c_j=0$
		for $j=N+1,\ldots, n-N-K$. In both cases, \eqref{deltas} leads to
		$$
		\sum_{j=1}^N
		\binom{n-N-k}{j}
		j!c_j= \sum_{j=1}^{n-N-k}
		\binom{n-N-k}{j}
		j!c_j=\delta^{(1)}_{n-N-k}.
		$$
		From this and \eqref{casi_tilde},
		$$
		\sum_{s=0}^{N-1}\alpha_{n,n-s}b_{n-s,i}+b_{n+1,i}=\frac{\delta_{n}^{(1)} \delta_{n-1}^{(1)} \ldots \delta_{n-N-k+1}^{(1)} \delta_{n-N-k}^{(1)} }{(N+k+1)!}=b_{n,i-1},
		$$
		as we wanted to prove.
		
	\end{proof}

	\section{Inverse problem}
	\label{inverso}
	In this section, we assume the sequence $\{\lambda_n\}\subset\mathbb{C}$ and 
	the sequence of monic polynomials $\{P_n\}, n\in\mathbb{N}$, are given,
	where 
	$\deg(P_n)=n$
	for each
	$n\in\mathbb{N}.$
	Then, we define a sequence $\{\delta_n^{(k)}\}, n\in\mathbb{N},$ as
	\begin{equation}
		\label{delta0}
		\delta_n^{(0)}=\lambda_n, \quad n=0,1,\dots,
	\end{equation}
	for $k=1,2,\dots,n$,	
	\begin{equation}
		\label{1*}
		\delta_n^{(k)}=(-1)^k\left|
		\begin{array}{cccccc}
			(\lambda_{n-k}-\lambda_{n-k+1}) b_{n-k+1,n-k}& (\lambda_{n-k}-\lambda_{n-k+2}) b_{n-k+2,n-k}&\cdots&\cdots  &(\lambda_{n-k}-\lambda_{n}) b_{n,n-k}\\ 
			1& b_{n-k+2,n-k+1}&\cdots &\cdots& b_{n,n-k+1}\\ 
			0 & 1&   \cdots& \cdots &b_{n,n-k+2} &\\
			\vdots  & 0  & \ddots &&\vdots \\
			\vdots  &\vdots & \ddots &\ddots &\vdots \\
			0 & 0 &\cdots & 1& b_{n,n-1}
		\end{array}
		\right|,
	\end{equation}
	and, for $k>n$, $\delta_n^{(k)}=0$.
	
	Now, we define a sequence of polynomials $\{a_n\}, n\in\mathbb{N}$, as in \eqref{1111} where
	\begin{equation}
		\label{def_an}
		a_{n,n-k}=\frac{1}{n!}\sum_{i=k}^{n} {n \choose i}(-1)^{n-i}\delta_i^{(k)}\quad \text{ for } k=0,1,\ldots, n\,.
	\end{equation}	
	
	Finally, consider the infinite order differential operator
	
	\begin{equation}
		\label{def_infop}
		L=\sum_{i=0}^{+\infty}a_i(x)\partial_x^i
	\end{equation}

	In the following result, we provide a necessary and sufficient condition for the sequences $\{\lambda_n\}$ and $\{P_n\}$ to be the eigenvalues and eigenpolynomials of a Bochner differential operator defined as in \eqref{1}.
	
	\begin{teo}
		Let $\{\lambda_n\}$ and $\{P_n\}$ be two sequences in the above conditions. Consider the sequences $\{\delta_n^{(k)}\}$ and $\{a_n\}$ defined in terms of $\{\lambda_n\}$ and $\{P_n\}$ as in \eqref{1*} and \eqref{def_an}, respectively. Consider the infinite differential operator $L$ defined in terms of $\{a_n\}$ as in \eqref{def_infop}. Let $N\in \mathbb{N}$. Then, the following statements are equivalent:
		\begin{enumerate}
			\item $a_i(x)=0\,, \forall i>N$, and the sequences $\{\lambda_n\}$ and $\{P_n\}$ satisfy \eqref{11} for $L$.
			\item  \eqref{eq-2} is fulfilled for $N$.
		\end{enumerate}
	\end{teo}
	
	\begin{proof}
		Theorems \ref{Lema2} and \ref{th-main} prove $1\Rightarrow 2$. Then, let us prove $2\Rightarrow 1$.
		
		First, let us prove the operator $L$ has order $N$. That is, we have to prove $a_{n,n-k}=0$ for any $n\geq N+1$ and $0\leq k\leq n$. Then, let $n\geq N+1$ and let $0\leq k\leq n$. From \eqref{def_an},
		
		$$n!a_{n,n-k}=\sum_{i=k}^n\binom{n}{i}(-1)^{n-i}\delta_i^{(k)}=
		\sum_{i=k}^N\binom{n}{i}(-1)^{n-i}\delta_i^{(k)}+\sum_{i=N+1}^n\binom{n}{i}(-1)^{n-i}\delta_i^{(k)}.$$
		Since we are assuming \eqref{eq-2} holds for $N$, we can use it for the above equation and obtain
		$$n!a_{n,n-k}=
		\sum_{i=k}^N\binom{n}{i}(-1)^{n-i}\delta_i^{(k)}+\sum_{i=N+1}^n\binom{n}{i}(-1)^{n-i}\sum_{\ell=k}^N(-1)^{N-\ell}\binom{i}{\ell}\binom{i-\ell-1}{N-\ell}\delta_\ell^{(k)}.$$
		The coefficient of $\delta_\ell^{(k)}$ in the right hand side of
		the above expression is
		\begin{align}
			\label{eq-aux-7}
			&\binom{n}{\ell}(-1)^{n-\ell}+\sum_{i=N+1}^n\binom{n}{i}(-1)^{n-i}(-1)^{N-\ell}\binom{i}{\ell}\binom{i-\ell-1}{N-\ell}\nonumber\\
			&=(-1)^{n-\ell}\frac{n!}{\ell!}\left[\frac{1}{(n-\ell)!}+\sum_{i=N+1}^n(-1)^{N-i}\frac{1}{i!(n-i)!}\frac{i!}{(i-\ell)!}\frac{(i-\ell-1)!}{(N-\ell)!(i-N-1)!}\right]\nonumber\\
			&=(-1)^{n-\ell}\frac{n!}{\ell!}\left[\frac{1}{(n-\ell)!}+\frac{1}{(N-\ell)!}\sum_{i=N+1}^n\frac{(-1)^{N-i}}{(i-\ell)(n-i)!(i-N-1)!}\right].
		\end{align}
		Let us consider now the sum $\sum_{i=N+1}^n\frac{(-1)^{N-i}}{(i-\ell)(n-i)!(i-N-1)!}$ in \eqref{eq-aux-7}. If we perform the change of variable $j=i-(N+1)$, we obtain
		\begin{align}
			\label{eq-aux-8}
			&\sum_{i=N+1}^n\frac{(-1)^{N-i}}{(i-\ell)(n-i)!(i-N-1)!}=
			\sum_{j=0}^{n-N-1}\frac{(-1)^{j+1}}{(j+N-\ell+1)(n-j-N-1)!j!}\nonumber\\
			&=\frac{-1}{(n-N-1)!}\sum_{j=0}^{n-N-1}\frac{(-1)^j}{j+N-\ell+1}\binom{n-N-1}{j}=\frac{-1}{(n-N-1)!}\frac{1}{(n-N)\binom{n-\ell}{n-N}}=-\frac{(N-\ell)!}{(n-\ell)!},
		\end{align}
		where the next-to-last equality follows from Lemma \ref{Lema33}. Replacing \eqref{eq-aux-8} in \eqref{eq-aux-7}, we obtain $a_{n,n-k}=0$.
		
		Now, we need to prove that $\{\lambda_n\}$ and $\{P_n\}$ are the eigenvalues and eigenpolynomials of the operator $L$ defined in terms of the sequence $\{a_n\}$. To do that, we need to check that \eqref{7} holds. Since $b_{n,m+k}=0$ if $m+k>n$, \eqref{7} is equivalent to
		\begin{equation}
			\label{eq-inv}
			\sum_{k=0}^{n-m}\delta_{m+k}^{(k)}b_{n,m+k}=\lambda_n b_{n,m}, \quad m=0,1,\dots,n.
		\end{equation}
		
		Let us prove \eqref{eq-inv} using induction on $n$. For $n=1$, $m$ takes values 0 and 1. For $m=0$, \eqref{eq-inv} is equivalent to $\delta_1^{(1)}=\lambda_1b_{1,0}$, which is true due to \eqref{1*}; for $m=1$, \eqref{eq-inv} is $\delta_1^{(0)}=\lambda_1$, which follows from \eqref{delta0}.
		
		Let $n\in\mathbb{N},\ n>1$, and assume that \eqref{eq-inv} takes place for $m=0,1,\dots,n$. Replacing $n$ by $n+1$, \eqref{eq-inv} becomes
		\begin{equation}
			\label{eq-inv1}
			\sum_{k=0}^{n-m+1}\delta_{m+k}^{(k)}b_{n+1,m+k}=\lambda_{n+1} b_{n+1,m}, \quad m=0,1,\dots,n.
		\end{equation}
		Let us prove \eqref{eq-inv1}. For $m=n+1$, \eqref{eq-inv1} is
		$$\delta_{n+1}^{(0)}=\lambda_{n+1},$$
		which is a direct consequence of \eqref{delta0}.
		
		\noindent	For $m=0,\dots,n$, let $s=n-m$. Then, we can write \eqref{eq-inv1} as
		\begin{equation}
			\label{eq-inv2}
			\sum_{k=0}^{s+1}\delta_{n-s+k}^{(k)}b_{n+1,n-s+k}=\lambda_{n+1} b_{n+1,n-s}, \quad s=0,1,\dots,n.
		\end{equation}
		Following \eqref{1*}, we can expand the determinant $\delta_{n+1}^{(s+1)}$ along its last column and, using the induction hypothesis, we reach
		$$\delta_{n+1}^{(s+1)}=(\lambda_{n+1}-\lambda_{n-s})b_{n+1,n-s}-\sum_{k=1}^s\delta_{n+k-s}^{(k)}b_{n+1,n+k-s},$$
		which is \eqref{eq-inv2}.
		
	\end{proof}

	\section{Conclusions and future work}
	\label{Conclusiones}
	We have introduced a sequence $\{\delta_n^{(k)}\}$ associated to a Bochner differential operator. Some properties of $\{\delta_n^{(k)}\}$ have been proven, including a (simplified) expression for the coefficients of the eigenpolynomials of a Bochner differential operator in terms of the elements of the sequences $\{\delta_n^{(k)}\}$ and an expression for
	$\{\delta_n^{(k)}\}$
	obtained from the $N-k+1$ first terms
	$\delta_k^{(k)},\,\delta_{k+1}^{(k)},\ldots, \delta_N^{(k)}$. Both properties combined lead to an expression for the coefficients of the eigenpolynomials in terms of a smaller amount of elements of the sequences $\{\delta_n^{(k)}\}$. In addition, the inverse problem has been studied, resulting in the derivation of a necessary and sufficient condition (that depends on $\{\delta_n^{(k)}\}$) for its solution.
	Furthermore, we have used some properties of $\{\delta_n^{(k)}\}$ to prove that the Bochner differential operator \eqref{Delta} gives rise to sequences of eigenpolynomials that satisfy an $(N+1)$-term recurrence relation. This approach is part of a larger goal: proving \cite[Conjecture 1.10]{Shapiro2}. To achieve this, two steps remain. First, we must prove that \eqref{operador_otro} also leads to sequences of eigenpolynomials that satisfy an $(N+1)$-term recurrence relation; second, we must prove that no other Bochner differential operator with that property exists. We hope to address this goal in a future work.

	\section*{Declaration of competing interests}
	
	There is no competing interest.

\end{document}